\documentclass[12pt]{article}
\usepackage{latexsym,amssymb,amsmath,amsfonts,amsthm}
\newtheorem{theorem}{Theorem}

\newtheorem{lemma}{Lemma}

\newtheorem{prop}{Proposition}

\newtheorem{assum}{Assumption}

\def\beq{ \begin{equation} }
\def\eeq{ \end{equation} }
\def\mn{\medskip\noindent}
\def\ms{\medskip}

\def\ep{\epsilon}

\def\square{\vcenter{\vbox{\hrule height .4pt
  \hbox{\vrule width .4pt height 5pt \kern 5pt
        \vrule width .4pt} \hrule height .4pt}}}

\def\RR{\mathbb{R}}
\def\ZZ{\mathbb{Z}}

\def\var{\hbox{var}\,}
\def\cov{\hbox{cov}\,}
\def\hbr{\hfill\break}
\def\TL{\mathbb{T}_L}
\def\clearp{}

\begin{document}

\title{Evolutionary Games on the Torus\\
with Weak Selection}
\author{J.T. Cox and Rick Durrett}

\maketitle

\begin{abstract} We study evolutionary games on the torus with $N$ points in dimensions $d\ge 3$.
The matrices have the form $\bar G = {\bf 1} + w G$,  where ${\bf 1}$ is a matrix that consists of all 1's,
and $w$ is small. As in Cox Durrett and Perkins \cite{CDP} we rescale time and space and take a limit as $N\to\infty$ and $w\to 0$.
If (i) $w \gg N^{-2/d}$ then the limit is a PDE on $\RR^d$. If (ii) $N^{-2/d} \gg w \gg N^{-1}$, then the limit is an ODE.
If (iii) $w \ll N^{-1}$ then the effect of selection vanishes in the limit.
In regime (ii) if we introduce a mutation $\mu$ so that $\mu /w  \to \infty$ slowly enough then we arrive at Tarnita's formula 
that describes how the equilibrium frequencies are shifted due to selection.
\end{abstract}

\section{Introduction}

Here we will be interested in $n$-strategy
evolutionary games on the torus $\TL = (\ZZ \bmod L)^d$. 
Throughout the paper we will suppose that $n\ge 2$ and  $d\ge 3$.
The dynamics are described by a game matrix
$G_{i,j}$ that gives the payoff to a player who plays strategy $i$ against an opponent who plays strategy $j$. 
As in \cite{CDP,spaceg}, we will study games with matrices of the form $\bar G = {\bf 1} + w G$, 
and ${\bf 1}$ is a matrix that consists of all 1's, and $w=\ep^2$. We use two notations for the small parameter to make it easier to
connect with the literature. 

There are two commonly used update rules. To define them introduce 

\begin{assum} \label{pass}
Let $p$ be a probability distribution on $\ZZ^d$  with finite range, $p(0)=0$ and that satisfies the following symmetry assumptions.
\begin{itemize}
\item
If $\pi$ is a permutation of $\{1, 2, \ldots d\}$ and $(\pi z)_i = z_{\pi(i)}$ then $p(\pi z)=p(z)$.
\item 
If we let $\hat z^i_i = - z_i$ and  $\hat z^i_j = z_j$ for $j\neq i$ then $p(\hat z^i)=p(z)$.
\end{itemize}
\end{assum} 

\noindent
If $p(z) = f(\|z\|_p)$ where $\|z\|_p$ is the $L^p$ norm on $\ZZ^d$ with $1\le p \le \infty$ then the symmetry assumptions are satisfied.

\mn
{\bf Birth-Death Dynamics.}
In this version of the model, a site $x$ gives birth at a rate equal to its fitness
$$
\psi(x) = \sum_{y} p(y-x) \bar G(\xi(x),\xi(y))
$$
and the offspring, which uses the same strategy as the parent, replaces a ``randomly chosen neighbor of $x$.'' Here, and in what follows, the phrase in quotes
means $z$ is chosen with probability $p(z-x)$. 

\mn
{\bf Death-Birth Dynamics.}
In this case, each site $x$ dies at rate 1 and is replaced by the offspring of a neighbor $y$ chosen with probability proportional to 
$p(y-x)\psi(y)$. 

\medskip
Tarnita et al.~\cite{TOAFN,TWN} have studied the behavior of evolutionary games on more general graphs when $w = o(1/N)$
and $N$ is the number of vertices. To describe their results, we begin with the two strategy game written as
\beq
\begin{matrix}
& 1 & 2 \\
1 & \alpha & \beta \\
2 & \gamma & \delta
\end{matrix}
\label{2sg}
\eeq
In \cite{TOAFN} strategy 1 is said to be favored by selection (written $1>2$) if
the frequency of 1 in equilibrium is $>1/2$ when $w$ is small. Assuming that

\ms
(i) the transition probabilities are differentiable at $w=0$,

\ms
(ii) the update rule is symmetric for the two strategies, and

\ms
(iii) strategy $1$ is not disfavored in the game given with $\beta=1$ and $\alpha=\gamma=\delta=0$

\mn
they argued that 

\mn
I. {\it $1> 2$ is equivalent to $\sigma \alpha + \beta > \gamma + \sigma \delta$
where $\sigma$ is a constant that only depends on the spatial structure and update rule.}

\medskip
In \cite{spaceg} it was shown that for games on $\ZZ^d$ with $d\ge 3$ that

\begin{theorem} \label{TF2}
I holds for the Birth-Death updating with $\sigma=1$ and 
for the Death-Birth updating with $\sigma = (\kappa+1)/(\kappa-1)$ where
\beq
\kappa = 1\left/ \sum_x p(x)p(-x) \right.
\label{kappadef}
\eeq
is the effective number of neighbors.
\end{theorem}

\noindent
The name for $\kappa$ comes from the the fact that if each $p(z) \in \{0,1/m\}$ for all $z$ then $\kappa=m$.

In \cite{TWN} strategy $k$ is said to be favored by selection in an $n$ strategy game if, in the presence of weak selection, its frequency 
is $> 1/n$. To state their result we need some notation.
\begin{align*}
\hat G_{*,*} = \frac{1}{n} \sum_{i=1}^n G_{i,i} \qquad & \hat G_{k,*} = \frac{1}{n} \sum_{i=1}^n G_{k,i} \\
\hat G_{*,k} = \frac{1}{n} \sum_{i=1}^n G_{i,k} \qquad & \hat G = \frac{1}{n^2} \sum_{i=1}^n \sum_{j=1}^n G_{i,j}
\end{align*}
where $*$'s indicate values that have been summed over. 
To make it easier for us to prove the result and to have
nicer constants, we will rewrite their condition
for strategy $k$ to be favored  as 
\beq
\alpha_1 ( \hat G_{k,*} - \hat G) + \alpha_2 (G_{k,k}  - \hat G_{*,*} ) + \alpha_3 (\hat G_{k,*} - \hat G_{*,k})  >0
\label{TFn}
\eeq
and refer to it as {\sl Tarnita's formula}.
The parameters $\alpha_i$ depend on the population structure
and the update rule,  
but they do not depend on the number of strategies or on the entries $G_{ij}$ of the payoff matrix.
In \cite{TWN} they divide by $\alpha_3$ so $\sigma_2=\alpha_1/\alpha_3$ and $\sigma_1 = \alpha_2/\alpha_3$.

When $n>2$, \eqref{TFn} is different from the 
result for almost constant sum three strategy games on $\ZZ^d$ proved in \cite{spaceg}. 
The condition \eqref{TFn} is linear in the entries in
the game matrix while the condition (8.13) in \cite{spaceg} for the
infinite graph is quadratic.  
This paper arose from an attempt to understand this discrepancy. The resolution, as we will explain, is that
the two formulas apply to different weak selection regimes.

The path we take to reach this conclusion is somewhat lengthy. In section 2, we introduce the voter model and
describe its duality with coalescing random walk. Section 3 introduces the voter model perturbations studied
by Cox, Durrett, and Perkins \cite{CDP}. Section 4 states their reesult that when space and time are scaled
appropriately, the limit is a partial differential equation. The limit PDE is then computed for birth-death
and death-birth updating. They are reaction diffusion equations with a reaction term that is a cubic polynoimal.  
Section 5 uses the PDE limit to analyze $2 \times 2$ games. Section 6 introduces a duality for voter model perturbations,
which is the key to their analysis. Section 7 gives our results for regimes (i) and (ii) in the abstract and
our version of Tarnita's formula for games with $n\ge 3$ strategies. Sections 8--11 are devoted to proofs. 

\clearp

\section{Voter model}

Our results for evolutionary games are derived from results for a more general class of processes called
{\it voter model perturbations}. To introduce those we must first describe the {\it voter model}.
The state of the voter model at time $t$ is $\xi_t : \ZZ^d \to S$ where $S$ is a finite set of states, and
$\xi_t(x)$ gives the state of the individual at $x$ at time $t$. To formulate this class of models, let
$p(z)$ be a probability distribution on $\ZZ^d$ satisfying the conditions in Assumption \ref{pass}. 
In the voter model, the rate at which the voter at $x$ changes its opinion from $i$ to $j$ is
$$
c^v_{i,j}(x,\xi) = 1_{(\xi(x)=i)} f_j(x,\xi), 
$$
where $f_j(x,\xi) = \sum_y p(y-x) 1(\xi(y)=j)$ is the fraction of neighbors of $x$ in state $i$.
In words at times of a rate 1 Poisson process the voter at $x$ wakes up and with probability $p(y-x)$ 
imitates the opinion of the individual at $y$.

To analyze the voter model it is convenient to construct the procees on a {\it graphical representation} introduced by Harris \cite{H76} and further developed by Griffeath \cite{G78}. For each $x \in \ZZ^d$ and $y$ with $p(y-x)>0$ let $T^{x,y}_n$, $n\ge 1$ be the arrival times of a Poisson process with rate $p(y-x)$. At the times $T^{x,y}_n$, $n \ge 1$, the voter at $x$ decides to change its opinion to match the one at $y$.  To indicate this, we draw an arrow from $(x,T^{x,y}_n)$ to $(y,T^{x,y}_n)$. To calculate the state of the voter model on a finite set, we start at the bottom and work our way up
determining what should happen at each arrow. A nice feature of this approach is that it simultaneously constructs the 
process for all initial conditions so that if $\xi_0(x) \le \xi'_0(x)$ for all $x$ then for all $t>0$ 
we have $\xi_t(x) \le \xi'_t(x)$ for all $x$.

\begin{figure}[ht]
\begin{center}
\begin{picture}(320,220)
\put(30,30){\line(1,0){260}}
\put(30,180){\line(1,0){260}}
\put(40,30){\line(0,1){150}}
\put(80,30){\line(0,1){150}}
\put(120,30){\line(0,1){150}}
\put(160,30){\line(0,1){150}}
\put(200,30){\line(0,1){150}}
\put(240,30){\line(0,1){150}}
\put(280,30){\line(0,1){150}}
\put(37,18){0}
\put(77,18){0}
\put(117,18){0}
\put(157,18){1}
\put(197,18){0}
\put(237,18){1}
\put(277,18){0}
\put(37,185){1}
\put(77,185){1}
\put(117,185){1}
\put(157,185){1}
\put(197,185){1}
\put(237,185){1}
\put(277,185){0}
\put(20,27){0}
\put(20,177){$t$}
\put(120,160){\line(1,0){40}}
\put(118,157){$<$}
\put(200,145){\line(1,0){40}}
\put(198,142){$<$}
\put(80,130){\line(1,0){40}}
\put(78,127){$<$}
\put(160,110){\line(1,0){40}}
\put(158,107){$<$}
\put(40,90){\line(1,0){40}}
\put(73,87){$>$}
\put(80,75){\line(1,0){40}}
\put(113,72){$>$}
\put(240,60){\line(1,0){40}}
\put(273,57){$>$}
\put(120,45){\line(1,0){40}}
\put(153,42){$>$}
\linethickness{1.0mm}
\put(120,180){\line(0,-1){50}}
\put(80,130){\line(0,-1){55}}
\put(120,75){\line(0,-1){30}}
\put(240,180){\line(0,-1){35}}
\put(200,145){\line(0,-1){35}}
\put(160,110){\line(0,-1){80}}
\end{picture}
\caption{Voter model graphical representation and duality}
\end{center}
\end{figure}
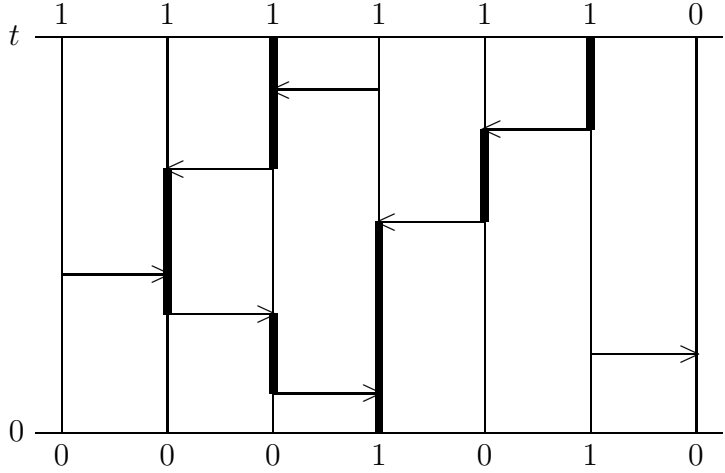

To define the {\it dual process}  we start with $\zeta^{x,t}_0 = x$ and work down the graphical representation. The process stays at $x$ until the first time $t-r$ that it encounters the tail of an arrow $x$. At this time, the particle jumps to the site $y$ at the head of the arrow, i.e., $\zeta^{x,t}_r = y$. The particle stays at $y$ until the next time the tail of an arrow is encountered and then jumps to the head of the arrow etc. 
Intuitively $\zeta^{x,t}_s$ gives the source at time $t-s$ of the opinion at $x$ at time $t$ so
$$
\xi_t(x) = \xi_{t-s}(\zeta^{x,t}(s)).
$$
The example in the picture should help explain the definitions.  The family of particles $\zeta^{x,t}_s$ are coalescing random walk.
Each particle at rate 1 makes jumps
according to $p$. If a particle $\zeta^{x,t}_s$ lands on the site occupied by $\zeta^{y,t}_s$ they coalesce to 1,
and we know that $\xi_t(x)=\xi_t(y)$. The dark lines indicate the locations of the two dual particles that coalesce.

Using duality it is easy to analyze the asymptotic behavior of the voter model. The results we
are about to quote were proved by Holley and Liggett \cite{HL}, and can also be found in Liggett's book \cite{L99}.
In dimensions 1 and 2,  random walks satisfying our assumptions are recurrent, so the voter model clusters, i.e.,
$$
P( \xi_t(x) \neq \xi_t(y) ) \le P( \zeta^{x,t}_t \neq \zeta^{x,t}_t ) \to 0.
$$  
In $d \ge 3$ random walks are transient so differences in
opinion persist. Consider, for simplicity, the case of two opinions, 0 and 1. Let $\xi^u_t$ be the voter model starting
from product measure with density $u$, i.e., the initial voter opinions are independent and $=1$ with probability $u$.
As $t \to\infty$, $\xi^u_t$ converges to a limit distribution $\nu_u$.

A consequence of this duality relation is that if we let $p(0|x)$ be the probability that
two continuous time random walks, one starting at the origin 0, and one starting at $x$ never hit then
$$
\nu_u ( \xi(0) = 1, \xi(x) = 0 ) = p(0|x) u(1-u)
$$
since in order for the two opinions to be different at time $t$, the corresponding random walks cannot hit, and they must land on
sites with the indicated opinions, an event of probability $u(1-u)$.

To extend this reasoning to three sites, consider random walks starting at 0, $x$, and $y$.
Let $p(0|x|y)$ be the probability that the three random walks never
hit and let $p(0|x,y)$ be the probability that the walks starting from $x$ and $y$ coalesce, but
they do not hit the one starting at 0. Considering the possibilities that the walks starting from $x$ and  $y$
may or may not coalesce:  
$$
\nu_u ( \xi(0) = 1, \xi(x) = 0 , \xi(y) = 0 )  = p(0|x|y) u(1-u)^2  + p(0|x,y) u(1-u).
$$  
All the finite dimensional distributions of $\nu_u$ can be computed in this way.

\clearp
 
\section{Voter model perturbations}

The processes that we consider have flip rates 
\beq
c^v_{i,j}(x,\xi) + \ep^2 h^\ep_{i,j}(x,\xi).
\label{frates}
\eeq
The perturbation functions $h^\ep_{ij}$, $j\ne i$, may be negative  but 
in order for the analysis in \cite{CDP} to work, there must be a 
law $q$ of $(Y^1, \ldots Y^M) \in (\ZZ^d)^M$ and functions $g_{i,j}^\ep \ge 0$, which 
converge to limits $g_{i,j}$ as $\ep\to 0$, so that for some $\gamma < \infty$, we have for $\ep\le \ep_0$
\beq
h^\ep_{i,j}(x,\xi) = - \gamma f_i(x,\xi) + E_{Y}[g_{i,j}^\ep(\xi(x+Y^1), \ldots \xi(x+Y^M))].
\label{vptech}
\eeq
In words, we can make the perturbation positive by adding a positive multiple of the voter flip rates.
This is needed so that \cite{CDP} can use $g^\ep_{i,j}$ to define jump rates of a Markov process.

For simplicity we will assume that both $p$ and $q$ are finite range. 
Applying Proposition 1.1 of \cite{CDP} now implies the existence of suitable $g^\ep_{i,j}$
and that all our calculations can be done using the original perturbation. However, to use Theorems 1.4 and 1.5 in \cite{CDP}
we need to suppose that 
\beq
h_{i,j} = \lim_{\ep\to 0} h^\ep_{i,j}.
\label{crate}
\eeq
has $|h_{i,j}(\xi) - h^\ep_{i,j}(\xi)| \le C \ep^r$ for some $r>0$, see (1.41) in \cite{CDP}.

\mn
{\bf Birth-Death Dynamics.}
If we let $r_{i,j}(0,\xi)$ be the rate at which the state of $0$ flips from $i$ to $j$, 
\begin{align}
r_{i,j}(0,\xi) & = \sum_{x} p(x) 1(\xi(x)=j) \sum_{y} p(y-x) \bar G(j,\xi(y))
\nonumber \\
 & = \sum_{x} p(x) 1(\xi(x)=j) \left( 1 + \ep^2 \sum_{k} f_k(x,\xi) G_{j,k} \right) 
\nonumber \\
& = f_j(0,\xi) + \ep^2 \sum_k f^{(2)}_{j,k}(0,\xi) G_{j,k},
\label{rijBD}
\end{align}
where $f^{(2)}_{j,k}(0,\xi) = \sum_x \sum_y p(x) p(y-x) 1(\xi(x)=j,\xi(y)=k)$. Thus the perturbation,
which does not depend on $\ep$ is
\beq
h_{i,j}(0,\xi) = \sum_{k} f^{(2)}_{j,k}(0,\xi) G_{j,k}.
\label{hBD}
\eeq
If $p$ is uniform on the nearest neighbors of 0, then $q$ is nonrandom and $Y^1, \ldots Y^m$
is a listing of the nearest and next nearest neighbors of 0.

\mn
{\bf Death-Birth Dynamics.}
Using the notation in \eqref{rijBD} the rate at which $\xi(0)=i$ jumps to state $j$ is
\begin{align}
\bar r_{i,j}(0,\xi) & = \frac{r_{i,j}(0,\xi)}{\sum_k r_{i,k}(0,\xi)}  
= \frac{f_j(0,\xi) + \ep^2 h_{i,j}(0,\xi)}{ 1 + \ep^2 \sum_k h_{i,k}(0,\xi) } 
\nonumber\\
& = f_j(0,\xi) + \ep^2 h_{i,j}(0,\xi) - \ep^2 f_j(0,\xi) \sum_k h_{i,k}(0,\xi) + O(\ep^4)
\label{rijDB}
\end{align}
The new perturbation, which depends on $\ep$, is
\beq
\bar h_{i,j}^\ep(0,\xi) = h_{i,j}(0,\xi) -f_j(0,\xi) \sum_k h_{i,k}(0,\xi) +O(\ep^2)
\label{hDB}
\eeq  
As noted above the technical condition \eqref{vptech} holds because $p$ has finite range.
\eqref{crate} holds with $r=2$.

\clearp

\section{PDE limit}

Let $\xi^\ep_t$ be the process with flip rates given in \eqref{frates}.
The next result is the key to the analysis of voter model perturbations on $\ZZ^d$.  
Intuitively, it says that if we rescale space to $\ep\ZZ^d$ and speed up time by $\ep^{-2}$ 
the process converges to the solution of a partial differential equation. The first thing we have to do is to define
the mode of convergence. Given $r\in(0,1)$, let $a_\ep = \lceil \ep^{r-1} \rceil \ep$, $Q_\ep = [0, a_\ep)^d$, and
$|Q_\ep|$ the number of points in $Q_\ep$. For $x \in a_\ep \ZZ^d$ and $\xi\in \Omega_\ep$ the space of all functions
from $\ep\ZZ^d$ to $S$ let
$$
D_i(x,\xi) = |\{ y \in Q_\ep : \xi(x+y) = i \}|/|Q_\ep|
$$ 

We endow $\Omega_\ep$ with the $\sigma$-field ${\cal F}_\ep$ generated by the finite-dimensional distributions. Given a sequence of measures
$\lambda_\ep$ on $(\Omega_\ep,{\cal F}_\ep)$ and continuous functions $w_i$, we say that $\lambda_\ep$ has asymptotic densities
$w_i$ if for all $0 < \delta, R < \infty$ and all $i\in S$
$$
\lim_{\ep\to 0} \sup_{x\in a_\ep\ZZ^d, |x| \le R} \lambda_\ep ( |D_i(x,\xi) - w_i(x)| > \delta ) \to 0
$$

\begin{theorem} \label{hydro}
Suppose $d \ge 3$.
Let $w_i: \RR^d \to [0,1]$ be continuous with $\sum_{i\in S} w_i = 1$. Suppose the initial conditions $\xi^\ep_0$
have laws $\lambda_\ep$ with local densities $w_i$ and let  
$$
u^\ep_i(t,x) = P( \xi^\ep_{t\ep^{-2}}(x) = i)
$$ 
If $x_\ep \to x$ then $u^\ep_i(t,x_\ep) \to u_i(t,x)$ the solution of the system of partial differential equations:
\beq
\frac{\partial}{\partial t} u_i(t,x) = \frac{\sigma^2}{2} \Delta u_i(t,x) +  \phi_i(u(t,x))
\label{PDElimit}
\eeq
with initial condition $u_i(0,x) = w_i(x)$. The reaction term 
\beq
\phi_i(u) = \sum_{j \neq i} \langle 1_{(\xi(0)=j)} h_{j,i}(0,\xi) -  1_{(\xi(0)=i)} h_{i,j}(0,\xi) \rangle_u
\label{phidef}
\eeq
where the brackets are expected value with respect to the voter model stationary distribution $\nu_u$
in which the densities are given by the vector $u$.
\end{theorem}

\noindent
This result is Theorem 1.2 in \cite{CDP}.
Intuitively, on the fast time scale the voter model runs at rate $\ep^{-2}$ versus the
perturbation at rate 1, so the process is always close to the voter equilibrium for 
the current density vector $u$. Thus, we can compute the rate of change of $u_i$ by assuming the nearby
sites are in that voter model equilibrium.

In a homogeneously mixing population the frequencies of the strategies in an evolutionary game follow the 
replicator equation, see e.g., Hofbauer and Sigmund's book \cite{HS98}: 
\beq
\frac{du_i}{dt}  = \phi^i_R(u) \equiv u_i \left( \sum_k G_{i,k} u_k - \sum_{j,k} u_j G_{j,k} u_k \right). 
\label{rep}
\eeq
We will now compute the reaction terms $\phi_i$ for our two examples.

\mn
{\bf Birth-Death Dynamics.} On $\ZZ^d$ we let $v_i$ be independent with $P(v_i = x ) = p(x)$. Let 
$$
 p_1 =  p(0|v_1|v_1+v_2) \quad\hbox{and}\quad p_2 =  p(0|v_1,v_1+v_2).
$$
In this case the limiting PDE in Theorem \ref{hydro} is 
$\partial u_i/\partial_t = (1/2d) \Delta u + \phi^i_B(u)$ where
\beq
\phi^i_B(u) =   p_1\phi^i_R(u) + p_2 \sum_{j\neq i}  u_i u_j (G_{i,i}-G_{j,i}+ G_{i,j}-G_{j,j}). 
\label{phiBi}
\eeq
See Section 12 of \cite{spaceg} for a proof.
Formula (4.8) in \cite{spaceg} implies that 
$$
2p(0|v_1,v_1+v_2) = p(0|v_1) - p(0|v_1|v_1+v_2),
$$
so it is enough to know the two probabilities on the right-hand side.

If coalescence is impossible then $p_1=1$ and $p_2=0$ and $\phi^i_B = \phi^i_R$.
There is a second more useful connection to the replicator equation. Let
$$
A_{i,j} = \frac{ p_2 }{ p_1   } (G_{i,i} + G_{i,j} - G_{j,i} - G_{j,j} ).
$$
The matrix is skew symmetric. That is, $A_{i,i}=0$ and if $i\neq j$ $A_{i,j} = -A_{j,i}$.
This implies $\sum_{i,j} u_i A_{i,j} u_j = 0$ and it follows that
$\phi^i_B(u)$ is $p_1$ times the RHS of the replicator equation for the game matrix $A+G$. 
This observation is due to Ohtsuki and Nowak \cite{ON06} who studied the limiting ODE
that arises from the pair approximation.

\mn
{\bf Death-Birth Dynamics.} On $\ZZ^d$ we let $v_i$ be independent with $P(v_i = x ) = p(x)$, let 
$$
\bar p_1 =  p(v_1|v_2|v_2+v_3) \quad\hbox{and}\quad \bar p_2 =  p(v_1|v_2,v_2+v_3).
$$
With Death-Birth updating the limiting PDE is $\partial u_i/\partial t = (1/2d) \Delta u + \phi^i_D(u)$ where
\begin{align}
\phi^i_D(u) &= \bar p_1 \phi^i_R(u)   + \bar p_2 \sum_{j\neq i}  u_i u_j (G_{i,i} -G_{j,i}+ G_{i,j} -G_{j,j}) 
\nonumber\\
& -(1/\kappa) p(v_1|v_2)  \sum_{j \neq i}  u_i u_j (  G_{i,j}-G_{j,i}).
\label{phiDi}
\end{align} 
where $\kappa = 1/P( v_1+v_2=0)$ is the ``effective number of neighbors.''
Again see Section 12 of \cite{spaceg} for a proof.
The first two terms are the ones in \eqref{phiBi}.
The similarity is not surprising since the numerators of the flip rates in \eqref{rijDB} are the flip rates in \eqref{rijBD}.
The third term comes from the denominator in \eqref{rijDB}. 
Formula (4.9) in \cite{spaceg} implies that 
$$
2p(v_1|v_2,v_2+v_3) = (1+1/\kappa) p(v_1|v_2) - p(v_1|v_2|v_2+v_3),
$$
so again it is enough to know the two probabilities on the right-hand side.

As in the Birth-Death case, if we let 
$$
\bar A_{i,j}  =  \frac{\bar p_2 }{\bar p_1  } (G_{i,i} + G_{i,j} - G_{j,i} - G_{j,j} ) 
- \frac{p(v_1|v_2)} {  \kappa \bar p_1 } (G_{i,j} - G_{j,i}),
$$
then $\phi^D_i(u)$ is $\bar p_1$ times the RHS of the replicator equation for $\bar A+G$.

\clearp

\section{Two strategy games}

In a homogeneously mixing population, the fraction of individuals playing the first strategy,
$u$, evolves according to the replicator equation \eqref{rep}:
\begin{align}
\frac{du}{dt} & = u \{ \alpha u + \beta (1-u) 
- u [ \alpha u + \beta (1-u) ] - (1-u) [ \gamma u + \delta (1-u) ] \} 
\nonumber\\
& = u(1-u)[\beta - \delta + \Gamma u ] \equiv \phi_R(u)
\label{rep2}
\end{align}
where we have introduced $\Gamma = \alpha - \beta - \gamma + \delta$. Note that $\phi_R(u)$ is a cubic
with roots at 0 and at 1. If there is a fixed point in $(0,1)$
it occurs at
\beq
\bar u =  \frac{\beta-\delta }{\beta- \delta  + \gamma - \alpha }
\label{fixp}
\eeq

Using results from the previous section gives the following.

\mn
{\bf Birth-Death Dynamics.} 
The limiting PDE is $\partial u/\partial t = (1/2d) \Delta u + \phi_B(u)$ where 
$\phi_B(u)$ is $p_1$ times the RHS of the replicator equation for the  game
\beq
\begin{pmatrix} \alpha & \beta + \theta \\ \gamma - \theta & \delta \end{pmatrix}
\label{2spert}
\eeq
and $\theta = (p_2/p_1)(\alpha+\beta-\gamma-\delta)$.

\mn
{\bf Death-Birth Dynamics.} 
The limiting PDE is $\partial u/\partial t = (1/2d) \Delta u + \phi_D(u)$ where
$\phi_D(u)$ is $\bar p_1$ the RHS of the replicator equation for the game in \eqref{2spert} but now
$$
\theta = \frac{\bar p_2}{\bar p_1}(\alpha+\beta-\gamma-\delta) - \frac{p(v_1|v_2)}{\kappa \bar p_1}(\beta-\gamma).
$$

\subsection{Analysis of $2 \times 2$ games}

Suppose that the limiting PDE is $\partial u/\partial t = (1/2d) \Delta u + \phi(u)$ 
where $\phi$ is a cubic with roots at 0 and 1. There are four possibilities

\begin{center}
\begin{tabular}{ccc}
$S_1$ & $\bar u$ attracting & $\phi'(0)>0$, $\phi'(1)>0$ \\
$S_2$ &$\bar u$ repelling  & $\phi'(0)<0$, $\phi'(1)<0$ \\
$S_3$ & $\phi<0$ on $(0,1)$   & $\phi'(0)<0$, $\phi'(1)>0$ \\
$S_4$ & $\phi>0$ on $(0,1)$   & $\phi'(0)>0$, $\phi'(1)<0$
\end{tabular}
\end{center}

\noindent
To see this, we draw a picture. For convenience, we have drawn the cubic as a piecewise linear function.

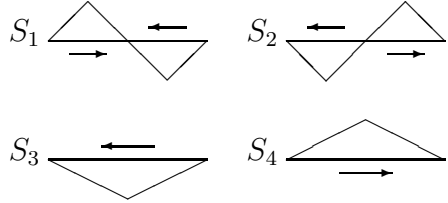
\begin{center}
\begin{picture}(220,105)
\put(15,75){$S_1$}
\put(30,75){\line(1,0){60}}
\put(30,75){\line(1,1){15}}
\put(45,90){\line(1,-1){30}}
\put(75,60){\line(1,1){15}}
\put(82,80){\vector(-1,0){14}}
\put(38,70){\vector(1,0){14}}
\put(105,75){$S_2$}
\put(120,75){\line(1,0){60}}
\put(120,75){\line(1,-1){15}}
\put(135,60){\line(1,1){30}}
\put(165,90){\line(1,-1){15}}
\put(142,80){\vector(-1,0){14}}
\put(158,70){\vector(1,0){14}}
\put(15,30){$S_3$}
\put(30,30){\line(1,0){60}}
\put(30,30){\line(2,-1){30}}
\put(60,15){\line(2,1){30}}
\put(70,35){\vector(-1,0){20}}
\put(105,30){$S_4$}
\put(120,30){\line(1,0){60}}
\put(120,30){\line(2,1){30}}
\put(150,45){\line(2,-1){30}}
\put(140,25){\vector(1,0){20}}
\end{picture}
\end{center} 

We say that {\it $i$'s take over} if for all $K$
$$
P( \xi_s(x) = i \hbox{ for all $x\in[-K,K]^d$ and all $s \ge t$}) \to 1 \quad\hbox{as $t\to\infty$.}
$$
Let $\Omega_0 = \{ \xi : \sum_x \xi(x) = \infty, \sum_x (1-\xi(x)) = \infty \}$
be the configurations with infinitely many 1's and infinitely many 0's. We say that
{\it coexistence occurs} if there is a stationary distribution $\mu$ for the spatial model with $\mu(\Omega_0)=1$. 
The next result follows from Theorems 1.4 and 1.5 in \cite{CDP}. The PDE assumptions and the other conditions
can be checked as in the arguments in Section I.4 of  \cite{CDP} for the Lotka-Volterra system. 
Intuitively, the result says that the behaviior of the particle system for small $\ep$ is the same as
that of the PDE.   

\begin{theorem} \label{CDPG}
If $\ep < \ep_0(G)$, then:\hbr 
In case $S_3$, 2's take over. In case $S_4$, 1's take over. \hbr
In case $S_2$, 1's take over if $\bar u < 1/2$, and 2's take over if $\bar u > 1/2$.\hbr
In case $S_1$, coexistence occurs. Furthermore, if $\delta>0$ and $\ep < \ep_0(G,\delta)$ then 
any stationary distribution with $\mu(\Omega_0)=1$ has
$$
\sup_x |\mu(\xi(x) =1 )- \bar u| < \delta.
$$
\end{theorem}

This result, after some algebra gives Tarnita's formula for two person games, Theorem \ref{TF2}.
The key observation is

\begin{lemma} \label{phipos}
$1 > 2$ if and only if the reaction term in the PDE has $\phi(1/2)>0$.
\end{lemma}

\begin{proof} Clearly $\phi(1/2)>0$ in case $S_4$ but not
  $S_3$. In case $S_1$, $\phi(1/2)>0$
implies $\bar u > 1/2$, while in case $S_2$
$\phi(1/2)>0$  implies $\bar u < 1/2$
and hence the 1's take over. 
\end{proof}

\noindent
With this result in hand, Theorem \ref{TF2} follows from the formulas for $\phi_B$ and $\phi_D$.

\clearp

\section{Duality for voter model perturbations} \label{sec:dualvmp}

Our next step is to introduce a duality that generalizes the one for the voter model.  Suppose now that we have a voter model perturbation of the form 
$$
h^\ep_{i,j}(x,\xi) = - \gamma f_i(x,\xi) + E_Y[g^\ep_{i,j}(\xi(x+Y^1), \ldots \xi(x+Y^M))]
$$
For each $x \in \ZZ^d$ and $y$ with $p(y)>0$ let $T^{x,y}_n$, $n\ge 1$ be the arrival times of a Poisson process with rate $p(y)$. At the times $T^{x,y}_n$, $n \ge 1$, $x$ decides to change its opinion to match the one at $x+y$ where the arithmetic is done modulo $L$ in each coordinate. We call this a voter event. 

To accommodate the perturbation we let 
$$
\|g^\ep_{i,j}\| = \sup_{\eta \in S^M} g^\ep_{i,j}(\eta_1, \ldots \eta_M)
$$ 
and introduce Poisson processes
$T^{x,i,j}_n$, $n\ge 1$ with rate $r_{i,j} = \ep^2 \|g^\ep_{i,j}\|$, and independent random variables $U^{x,i,j}_n$, $n\ge 1$
uniform on $(0,1)$. At the times $t=T^{x,i,j}_n$ we draw arrows from $x$ to $x+Y^i$ for $1\le i \le M$. 
We call this a branching event.
If $\xi_{t-}(x)=i$ and 
\beq
g_{i,j}^\ep(\xi_{t-}(x+Y^1), \ldots \xi_{t-}(x+Y^m)) < r_{i,j}U^{x,i,j}_n
\label{jrule}
\eeq
then we set $\xi_t(x)=j$. The uniform random variables slow down the transition rate from the maximum possible rate $r_{i,j}$ to the one appropriate
for the current configuration.

To define the dual, suppose we start with particles at $X_1(0), \ldots X_k(0)$ at time $t$. We let $K(0)=k$ be the number of particles,
$J(0) = \{1, 2, \ldots k\}$ be the indices of the active particles, and $T_0=0$. Suppose we have constructed the
dual up to time $T_m$ with $m \ge 0$. No particle moves from its position at time $T_m$ until the first time $r > T_{m}$ 
that the tail of an arrow touches one of the active particles at time $t-r$. Call that time $T_{m+1}$.
We extend the definitions of $K(t)$, $X_i(t)$, $i \le K(t)$, and $J(t)$ to be constant on $[T_m,T_{m+1})$.

If the arrow is from a voter event affecting particle number $i$ then $X_i$ jumps to the head of the arrow at time $T_{m+1}$. 
If there is another active particle $X_j$ on that site, the two coalesce to 1 and the higher numbered particle
is removed from the active set at time $T_{m+1}$. If the event is a branching event, we add new particles numbered $K(T_m)+k$ 
at $X_i(T_m)+Y^k$ for $1\le k \le M$ and set $K(T_{m+1})=K(T_m)+M$. If there are collisions between the newly created particles
and existing active particles, those newly created particles are not added to the active set. Our proof will show that in
the situation covered in Theorem \ref{odeconv} the probability of a collision at a branching event will go to zero as $N\to\infty$

Durrett and Neuhauser \cite{DN} called $I(s) = \{ X_i(s) : i \in J(s) \}$ the influence set because 

\begin{lemma}
If we know the values of $\xi_{t-s}$
on $I(s)$, the locations and types of arrows that occurred at the jump times $T_m \le s$, and the associated uniform random variables $U_m$ then we can compute the values of $\xi_t$ at $X_1(0), \ldots X_k(0)$ by working our way up the graphical representation
starting from time $t-s$ and determining the changes that should be made in the configuration at each jump time.
\end{lemma} 

This should be clear from the construction. A formal proof can be found in Section 2.6 of \cite{CDP}.
The computation process, as it is called in \cite{CDP} is complicated, but is useful because
up to time $t/\ep^2$ there will only be $O(1)$ branching events. In between these events there will be many random
walk steps that on the rescaled lattice will converge to Brownian motions. 
 
\clearp

\section{Results for the torus} \label{sec:torus}

To motivate the results that we are about to state, recall that if we have a random walk on the torus $\TL = (\ZZ \bmod L)^d$
that takes jumps at rate 1 with a distribution $p$ that satisfies our assumptions in Assumption \ref{pass}, then: 

\begin{itemize}
\item
One random  walk needs of order $L^2$ steps to converge to the uniform distribution on the torus (see Proposition \ref{rw1} in Section \ref{sec:pf4}).
\item
Two independent random walks starting from randomly chosen points will need of order $N=L^d$ steps to meet
for the first time. See e.g., \cite{Cox89}
\end{itemize}

\subsection{Regime 1. $\ep_L^{-1} \ll L$, or $w \gg N^{-2/d}$}

In this case when we rescale space by multiplying by $\ep_L$ 
then the limit of the torus is all of $\RR^d$ and the PDE limit, Theorem \ref{hydro}, holds. Thus, 
one can apply results from Section 7 in \cite{CDP} to show that the conclusions of Theorem \ref{CDPG} hold in cases $S_2-S_4$. 
Indeed, since we are on the torus the linearly growing ``dead zone'' produced by the block construction eventually coves the 
entire torus and the weaker type becomes extinct at a time $O(L)$.

Case $S_1$ is more interesting. We cannot have a stationary distribution since we are dealing with a Markov chain
on a finite set in which $\xi(x)=i$ for all $x$ are absorbing. However, as is the case for many other
particle systems, e.g., the contact process on a finite set, we will have a quasi-stationary distribution that will persist 
for a long time. Using the comparison with oriented percolation described in Chapter 6 of \cite{CDP} and in \cite{Dur95}  
we can show 

\begin{theorem} \label{expsurv}
Consider a two strategy evolutionary game in case 1, so $\phi(u) = \lambda u(u-\rho)(1-\rho)$.
Suppose that $\ep_L^{-1} \sim C L^\alpha$ where $0<\alpha<1$ and that for each $L$ we start from a product measure in which 
each type has a fixed positive density.  Let $N_1(t)$ be the number of sites occupied by 1's at time $t$. 
There is a $c > 0$ so that for any $\delta > 0$ if $L$ is large and
$\log L \le t \le \exp( c L^{(1-\alpha)d})$ then $N_1(t)/N \in (\rho-\delta,\rho+\delta)$ with high probability. 
\end{theorem}

The $\log L$ time needed to come close to equilibrium could be replaced by a fixed time $T$ that depends on $\lambda$, $\rho$, $\delta$,
and the initial density of 1's. Our proof will show that with high probability at any time $\log L \le t \le \exp( c L^{(1-\alpha)d})$ 
the density is close to $\rho$ (in the sense used to define the hydrodynamic limit) across most of the torus.

In many situations, e.g., the supercritical contact process on the $d$-dimensional cube \cite{TM2}, and power-law random graphs \cite{MMVY}, 
 the quasi-stationary distribution persists for time $\exp(\gamma N^d)$.
However, we think that is not true in Theorem \ref{expsurv}.  For a simpler situation where we can prove this, 
consider the 

\mn
{\bf Contact process with fast voting,} studied by Durrett, Liggett, and Zhang \cite{DLZ}. 
In this voter model perturbation, there are two states, 0 and 1. 

\begin{itemize}
  \item 
$h_{1,0}(x,\xi) \equiv 1$: particles die at rate 1. 
\item
$h_{0,1}(x,\xi) = \lambda f_1(x,\xi)$: a particle at $x$ gives birth to a new one at $x+y$ at rate $\lambda p(y)$.
\end{itemize}

We only have to keep track of 1's so the reaction term
\begin{align*}
\phi(u) \equiv \phi_1(u) & = \langle 1_{(\xi(0)=0)} \lambda f_1(x,\xi) - 1_{(\xi(0)=1)} \rangle_u  \\
& = \lambda p(0|v_1) u(1-u) - u
\end{align*}
If $\beta = \lambda p(0|v_1) > 1$ then 0 is an unstable equilibrium for $du/dt = \phi(u)$
and there is a fixed point at $\rho = (\beta-1)/\beta$.
The proof of Theorem \ref{expsurv} can be easily extended to show survival up to time $\exp(cL^{(1-\alpha)d})$ with high
probability. In this case we can prove a partial converse.

\begin{theorem} \label{cpdie}
There is an $C<\infty$ so that the system dies out by time
$\exp(C L^{d-2\alpha} \log L)$ with high probability.
\end{theorem}

\noindent
Note that the powers of $L$ in the two results, $d(1-\alpha)$ and $d-2\alpha$, do not match. We suspect that the larger value is the
correct answer. However it is not clear how to improve the proof of Theorem \ref{expsurv} to close the gap.

\subsection{Regime 2. $L \ll \ep_L^{-1} \ll L^{d/2}$ or $N^{-2/d} \gg w \gg N^{-1}$.}

In this case the time scale for the perturbation to have an effect, $\ep_L^{-2}$ is much larger than the time $O(L^2)$
needed for a random walk to come to equilibrium, but much smaller than the time $O(L^d)$ it takes for two random walks to hit. Because of this, 
the particles in the dual will (except for times $O(L^2\log L)$ after the initial time or a branching event) will be approximately independent and uniformly distributed across the torus. Thus, if we speed up time  by $\ep_L^{-2}$ the fraction of sites on the torus in state $i$ will converge to an ordinary differential equation. To formulate a precise result define the empirical density by
$$
U_i(t) = \frac{1}{N} \sum_{x \in {\cal T}_L} 1\left( \xi^\ep_{t\ep^{-2}_L}(x) = i \right)
$$

\begin{theorem} \label{odeconv}
Suppose that $L^2 \ll \ep_L^{-2} \ll L^d$. If $U_i(0) \to u_i$ then $U_i(t)$ converges uniformly  on compact sets to $u_i(t)$, the solution of
$$
\frac{du_i}{dt} = \phi_i(u) \qquad u_i(0) = u_i
$$
where $\phi_i$ is the reaction term in \eqref{phidef}.
\end{theorem}

Thus in Regime 2, we have ``mean-field'' behavior, but the reaction function in the ODE is computed using the voter model equilibrium, not the product measure that is typically used in heuristic calculations. The asymptotic behavior of the particle system is now the same as that of the limiting ODE.
In particular in case $S_2$, it will converge to 0 or 1, depending on whether the initial density $u_1 < \bar u$ or $u_1 > \bar u$. 

\subsection{Tarnita's formula}

Suppose each individual switches to a strategy chosen at random from the $n$ possible strategies at rate $\mu$.

\begin{theorem} \label{TFth}
Suppose that $N^{-2/d} \gg w \gg N^{-1}$. If $\mu \to 0$ and $\mu/w \to\infty$ slowly enough, then in
an $n$-strategy game strategy $k$ is favored by mutation if and only if
$$
\phi_k(1/n, \ldots, 1/n) > 0.
$$
\end{theorem} 

\noindent 
Note the similarity to Lemma \ref{phipos}. Intuitively,  the change from uniformity will be due to lineages that have one branching event. 
We do not claim that these conditions are necessary for the conclusion to hold but they are needed for out proof to work. 
Our next step is to show that we recover the formula in \cite{TWN} and identify the coefficients.

\mn
{\bf Birth-Death Dynamics.} 
In this case the limiting PDE is 
$\partial u_k/\partial_t = (1/2d) \Delta u + \phi_i^B(u)$ where
$$
\phi_k^B(u) =   p_1\phi_k^R(u) + p_2 \sum_{j}  u_k u_j (G_{k,k}-G_{j,k}+ G_{k,j}-G_{j,j}). 
$$
see \eqref{phiBi}.
If we take $u_i \equiv 1/n$ then 
\begin{align*}
p_1 \phi_k^R(1/n,\ldots 1/n) & = \frac{p_1}{n} \left( \sum_i G_{k,i} \frac{1}{n} - \sum_{i,j} \frac{1}{n} G_{i,j} \frac{1}{n} \right) \\
& = \frac{p_1}{n} (\hat G_{k,*} - \hat G)
\end{align*}
while the second term in \eqref{phiBi} is 
$$
\frac{p_2}{n}  (G_{k,k} - \hat G_{*,k} + \hat G_{k,*} - \hat G_{*,*})
$$
For Birth-Death dynamics \eqref{TFn} holds with $\alpha_1 = p(0|e_1|e_1+e_2)$ and $\alpha_2 = \alpha_3 = p(0|e_1,e_1+e_2)$.

\mn
{\bf Death-Birth Dynamics.} 
In this case the limiting PDE is $\partial u_k/\partial t = (1/2d) \Delta u + \phi_k^D(u)$ where
\begin{align}
\phi_k^D(u) &= \bar p_1 \phi_k^R(u)   + \bar p_2 \sum_{j}  u_k u_j (G_{k,k} -G_{j,k}+ G_{k,j} -G_{j,j}) 
\nonumber\\
& -(1/\kappa) p(v_1|v_2)  \sum_{j}  u_k u_j ( G_{k,j} - G_{j,k}).
\end{align}
see \eqref{phiDi}.
The computations for the first two terms are as in Birth-Death case with $p_i$ replaced by $\bar p_i$.
The third term is
$$
 -\frac{(1/\kappa) p(v_1|v_2)}{n} ( \hat G_{k,*} - \hat G_{*,k}).
$$
Thus for Death-Birth dynamics \eqref{TFn} holds with
$\alpha_1 = p(e_1|e_2|e_2+e_3)$, $\alpha_2 = p(e_1|e_2,e_2+e_3)$, 
$\alpha_3 = p(e_1|e_2,e_2+e_3) - (1/\kappa)p(e_1|e_2)$, where $1/\kappa = P(e_1+e_2=0)$ is the effective number of neighbors.

\medskip
These calculations for Theorem \ref{TFth} apply to graphs other than the torus. For example, a random $r$-regular graph looks locally like a tree
in which each vertex has $r$ neighbors. Of course the values of the constants for the random regular graph will be different from those
on the torus.  

\clearp

\section{Proof of Theorem \ref{expsurv}}

\begin{proof}
We will prove only the asymptotic lower bound on the number of 1's. Once this is done, the upper bound follows by interchanging the roles of 0's and 1's.
The reaction term $\phi(u) = \lambda u(u-\rho)(1-u)$ so the limiting PDE satisfies Assumption 1 in CDP with $u_* = u^* = \rho$.

\mn
There are constants $0< v_0 < \rho < v_1 < 1$ and $w,L_i>0$ so that

\mn
(i) if $u(0,x) \ge v_0$ when $|x| \le L_0$ then $\liminf_{t \to\infty} \inf_{|x| \le wt} u(t,x) \ge \rho$.

\mn
(ii) if $u(0,x) \le v_1$ when $|x| \le L_1$ then $\limsup_{t \to\infty} \sup_{|x| \le wt} u(t,x) \le \rho$.

\mn
In our case if $w$ is chosen small enough we can take the $v_i$ and $L_i$ to be any positive numbers.
See Aronson and Weinberger \cite{AW1,AW2}. We will take $v_0 = \min\{\rho/2, u_1/2\}$ where $u_1$
is the density of 1's in the initial product measure.

As in the proof of Theorem 1.4 in CDP on the infinite lattice, we use a block construction. 
We let $K$ be the largest odd integer so that we can fit $K^d$ adjacent cubes with  sides $ = 2\ep_L^{-1}$ into the torus.
Asymptotically we have $K \sim (C/2)L^{1-\alpha}$. 
Suppose that the origin is in the middle of one of the blocks and call that box $I_0$.
The other blocks can be indexed by  $\{-K/2, -K/2+1, \ldots K/2\}^d$. There is some space leftover outside
our blocks, so the block construction lattice is not a torus but a flat cube.

To achieve the PDE limit we scale space by multiplying by $\ep_L$ and speed up time by $\ep_L^{-2}$. 
To define our block event we consider the initial condition for the PDE in which $u(0,x) \ge v_0$ when $|x| \le 1/2$.
Given $\delta > 0$, we choose $T$ large enough so that $u(T,x) \ge \rho - \delta/2$ when $|x| \le 3$.   

As in the hydrodynamic limit, given $r\in(0,1)$, let $a_\ep = \lceil \ep^{r-1} \rceil \ep$, $Q_\ep = [0, a_\ep)^d$, and
$|Q_\ep|$ the number of points in $Q_\ep$. For $x \in a_\ep \ZZ^d$ and $\xi\in \Omega_\ep$ the space of all functions
from $\ep\ZZ^d$ to $S$ let
$$
D_i(x,\xi) = |\{ y \in Q_\ep : \xi(x+y) = i \}|/|Q_\ep|
$$ 
We say that the configuration in the box $I_k$ is good at time $t$ if in each small box $xa_\ep + Q_\ep$ contained
in $k+[-1/2,1/2]^d$ the density $D_i(x,\xi_t) \ge v_0$, and it is very good if in each small box $xa_\ep + Q_\ep$ contained
in $k+[-1,1]^d$ we have $D_i(x,\xi_t) \ge \rho-\delta$.  It follows from the hydrodynamic limit that 

\begin{lemma}
Let $\theta>0$. Suppose the configuration in $I_k$ is good at time $mT$. If $L$ is large enough then with probability $\ge 1-\theta/2$ 
all of the boxes $I_\ell$ with $\ell - k  \in \{-1,0,1\}^d$ are very good at time $(m+1)T$.
\end{lemma}

Determining whether that all of the boxes $I_k$ with $k \in \{-1,0,1\}^d$ are very good at time $T$ can be done by
running the dual processes from all of these points back to time 0. Bounding the dual by a branching random walk and then using 
exponential estimates it is easy to show:

\begin{lemma}
Given $\theta >0$ and $T$, there is an $C$ so that if $L$ is large enough with probability $\ge 1-\theta/2$
none of the duals starting in $[-3,3]^d$ escape from $[-3-CT,3+CT]^d$ by time $T$.
\end{lemma}

If the configuration in $I_k$ is good at time $mT$, all of the boxes $I_\ell$ with $\ell - k  \in \{-1,0,1\}^d$ are very good at time $(m+1)T$,
and none of the dual processes escape from $k + [-3-CT,3+CT]^d$ as we work backwards from time $(m+1)T$ to time we set $\eta(k,m)=1$.
If we fail to achieve any one of our goals we set $\eta(k,m)=0$. If the configuration in $I_k$ is not good at time $mT$ we define 
$\eta(k,m)$ by an independent Bernoulli that is 1 with probability $1-\theta$.

These variables $\eta(k,m)$ define for us an oriented site percolation process on the graph $(Z \bmod K)^d \times \{0,1,2,\ldots\}$
in which $(k,m)$ is connected to $(\ell,m+1)$
when  $\ell - k  \in \{-1,0,1\}^d$. Writing $z$ as shorthand for $(k,m)$  the collection of $\eta(z)$ is ``$M$
dependent with density at least $1-\theta$'' which means that for any $k$,
\begin{align} 
\label{Mdepdef}
P(&\eta(z_i) = 1|\eta(z_j),j\neq i) \ge (1-\theta),\\
\nonumber&\hbox{ whenever }z_i , 1\le
i\le k\hbox{ satisfy $|z_i-z_j|> M$ for all $i\ne j$.}
\end{align}

It is typically not difficult to prove results for $M$-dependent percolation processes with $\theta$ small (see
Chapter 4 of \cite{Dur95}), but here it will be useful to simplify things
by applying Theorem~1.3 of Liggett, Schonmann, and Stacy \cite{LSS} to reduce to the case of independent percolation.  
By that result, under \eqref{Mdepdef}, there is a constant
$\Delta$ depending on $d$ and $M$ such that if
\begin{equation}
\label{LSStheta'}
1-\theta' = \Bigr(1- \frac{\theta^{1/\Delta}}{(\Delta-1)^{(\Delta-1)/\Delta}}\Bigl) \bigl(1-(\theta(\Delta-1))^{1/\Delta}\bigr),
\end{equation}
we may couple $\eta(z)$ with a family
$\zeta(z)$ of iid Bernoulli random variables
with $P(\zeta(z)=1)=1-\theta'$ such that $\zeta(z)\le \eta(z)$ for all $z$. 

With this result in hand we can prove the long time survival of our process on the torus by using 

\begin{lemma}
Suppose that $\theta < \theta_0$. Start the oriented percolation on $[-K/2,K/2]^d$ with all sites occupied 
and let $\tau$ be the first time all sites are vacant. There is a constant $c$ so that as $K\to\infty$
\beq
P( \tau > \exp(cK^d) ) \to 1 
\label{explb}
\eeq
\end{lemma}
 
Mountford \cite{TM2} proved for the contact process on $[1,N]^d$ that for all $\lambda>\lambda_c$
\beq
(\log E\tau )/N^d \to\gamma
\label{expgrow}
\eeq
Earlier he showed, see \cite{TM}
\beq
\tau /E\tau \Rightarrow {\cal E}
\label{explim}
\eeq
where ${\cal E}$ has an exponential distribution with mean 1.

\begin{proof}[Proof of \eqref{explb}]
We claim that our result can be proved by the method Mountford used 
to prove the sharp result in \eqref{expgrow} for all $\lambda > \lambda_c$. To explain why the reader should believe this,
we note that Lemma 1.1 in \cite{TM2} concerns connectivity properties of an oriented percolation process in which sites are open  
with probability $1-\ep_0$ which is then extended to the contact process with $\lambda > \lambda_c$ by using the
renormalization argument of Bezuidenhout and Grimmett \cite{BezGri}.

To be specific, Mountford, who only gives the details in $d=2$, shows that if $\lambda>0$ there are constants $L$ and $c>0$ 
so that for the contact process $\bar\xi^x_t$ in $D_n = [1,n] \times [1,4L]$  starting with a single occupied site at $x$ 
$$
\inf_{\lambda n^2 \le t \le n^8} \inf_{x,y\in D_n} P( \bar\xi^x_t(y) = 1 ) \ge c > 0
$$
See his Corollary 1.1. This result which is the key to the proof is also true for our oriented percolation process.
\end{proof}  

The last step in the proof of Theorem \ref{expsurv} is to note that if we start with product measure with a density $u_1$ of 1 then with high probability all
of the $I_m$ are good at time 0. To do this we note that the number of small boxes in the torus is a polynomial in $L$ but the probability
of an error in one is $\le \exp(-cL^{(1-r)d})$ where $r$ is the constant used to define the sizes of the little boxes.
\end{proof}

To justify the remark after Theorem \ref{expsurv} we use Proposition 1.2 of Mountford \cite{TM}. Write $\xi^1_t$ for the contact process in $[1,n]^2$
starting from all sites occupied. He shows that there are sequences $a(n), b(n) \to \infty$
with $b(n)/a(n) \to \infty$ so that $P^1(\tau < b(n) ) \to 0$ and 
$$
\inf_{\xi_0} P^{\xi_0} (\xi^1_{a(n)} = \xi_{a(n)} \hbox{ or } \tau < a_n ) \to 1
$$
In words the process either dies out before time $a(n)$ or at time $a(n)$ agrees with the process starting from all 1's. 
This idea, which is due to Durrett and Schonmann \cite{DurSch} allows one to prove the limit is exponential by showing that it has the
lack of memory property. Unfortunately,  Mountford writes $\sup$ rather than
$\inf$ in the conclusion. He cannot mean $\sup$ because that is attained by $\xi_0\equiv 1$ and the probability is 0. 

\clearp

\section{Proof of Theorem \ref{cpdie}} \label{sec:cpth}

\begin{proof}
Suppose  for simplicity that $\ep_L^{-1} \sim L^\alpha$.
Start a coalescing random walk $\bar\zeta^L_t$ with one particle at each site of the torus.

\begin{lemma}
Let $\bar N_L(t)$ be the number of particles in the coalescing random walks at time $t$. 
There is a constant $C_1$ so that for all $L$, 
$$
E\bar N_L(t) \le C_1/(1+t)\quad\hbox{for all $0\le t \le L^{2\alpha}$}.
$$
With high probability, i.e., one that tends to 1 as $L\to\infty$, 
$$
\bar N_L(t) \le 4C_1N/(1+t)\quad\hbox{for all $0\le t \le L^{2\alpha}$}.
$$
\end{lemma}

\noindent
The constant $C_1$ is special but all the others $C$'s are not and will change from line to line.

\begin{proof}
Let $S_t$ be the first coordinate of our $d$-dimensional random walk that makes jumps according to $p$ at rate 1. 
$$
E\exp(\theta S_t) = \exp(t(\phi(\theta)-1)) \quad\hbox{where}\quad \phi(\theta) = \sum_{z} e^{\theta z_1} p(z)
$$
Since $z \to \exp(\theta z_1)$ is convex and $p$ has mean zero and finite range
$$
0 \le \phi(\theta) - 1  \sim \frac{\sigma^2 \theta^2}{2} \quad\hbox{as $\theta \to 0$}.
$$
Let $1/2 < \rho < 1$. If we take $\theta = t^{\rho-1}/2\sigma^2$ and $t$ is large we have 
$$
P( S_t > t^{\rho} ) \le \exp\left( - \theta t^{\rho} + \sigma^2\theta^2 t \right) \le \exp( - t^{2\rho-1}/4\sigma^2 ).
$$

Let $\beta\in(\alpha,1)$ so that $\rho=\beta/2\alpha < 1$, and let $\delta = 2\rho-1$. Using the last result on each of the $d$ coordinates.

\mn 
(*) With probability $\ge 1 - N \cdot d\exp(- L^{\delta}/4\sigma^2 )$ no particle that starts outside $[-2L^{\beta},2L^\beta]^d$ will enter $[-L^\beta,L^\beta]^d$
by time $L^{2\alpha}$ and no particle that starts inside $[-2L^\beta,2L^\beta]^d$ will exit $[-3L^\beta,3L^\beta]^d$. 

\mn
Let $\bar\zeta_t$ be the coalescing random walk on $\ZZ^d$, let $p(t) = P( x \in \bar\zeta_t)$.  
and let $p_L(t) = P( x \in \bar\zeta^L_t )$. The last two probabilities do not depend on $x$ by translation invariance.
(*) implies that
$$
|p_L(t) - p(t)| \le  N \cdot d\exp(- L^{\delta}/4\sigma^2 ) \quad\hbox{for all $t \le L^{2\alpha}$}. 
$$
The result for the expected value now follows from a result of Bramson and Griffeath \cite{BG} that shows $p(t) \sim c/t$ as $t\to\infty$. 

We begin by proving the second result for each fixed $t$. A result of Arratia \cite{Arr}, see Lemma 1 on page 913, shows 
$$
P( x \in \bar\zeta_t, y \in \bar\zeta_t ) \le P( x \in \bar\zeta_t) P(  y \in \bar\zeta_t ),
$$
so we have
$$
\var(\bar N^L_t) \le N p_L(t)(1-p_L(t)) \le E\bar N_L(t).
$$

Using Chebyshev's inequality, and $\bar N_L(t) \le C_1N/(1+t)$
$$
P( |\bar N_L(t) - E\bar N_L(t)| \ge C_1 N/(1+t) ) \le \frac{1+t}{C_1N} \le \frac{C}{L^{d-2\alpha}}.
$$
To complete the proof now let $M = [\log_2 L^{2\alpha}]$. Applying the last result to the $O(\log L)$ values
$s_i = 2^i$ with $1\le i\le M$ we see that 
$$
P( \bar N_L(s_i) \le 2C_1N/(1+t) \hbox{ for }1\le i \le M ) \ge 1- \frac{C \log L}{L^{d-2\alpha}}.
$$
Using the fact that $t \to \Bar N_L(t)$ is decreasing and  $s_{i+1}/s_i=2$ we have
$$
P( \bar N_L(t) \le 2C_1 N/(1+t/2) \hbox{ for }0 \le t \le L^{2\alpha} ) \ge 1- \frac{C \log L}{L^{d-2\alpha}},
$$
which proves the desired result. 
\end{proof}

The dual for the contact process with fast voting is a branching coalescing random walk. The maximum branching rate
is $\|h\|_\infty \ep_L^2$, so the expected number of branchings that occur on the space-time set covered by
particles in the coalescing random walk is
$$
\le 4C_1||h\|_\infty \ep_L^2 \int_0^{\ep_L^{-2}} N/(1+t) \, dt  \le C L^{d-2\alpha} \log L.
$$
Since the total number of branchings on the space-time set occupied by particles is Poisson, the probability that no branching occurs is
$$
\ge \exp( - C L^{d-2\alpha} \log L ).
$$
Since there are deaths in the graphical representation for the contact process, the rest is easy.
With probability $\ge \exp( - C L^{d-2\alpha} )$ all of  the particles at time $L^{2\alpha}/2$ will be hit 
by a death by time $L^{2\alpha}$. Thus even if the process is in the all 1's state at time 0,
the probability it dies out by time $L^{2\alpha}$ is at least 
$$
\ge \exp( - C_2 L^{d-2\alpha} \log L ) \equiv M.
$$
If we are given $M^2 = \exp( 2C_2  L^{d-2\alpha} \log L )$ trials the probability we always fail is
$$
\le \left( 1 - 1/M \right)^{M^2} \le \exp(-M).
$$
This completes the proof.
\end{proof}

We could try to prove the upper bound on the survival tiem by arguing that with probability $\ge \exp( - C L^{d-2\alpha} )$ all the particles at time $L^{2\alpha}$ 
land on a 0 in the configuration at time $t-L^{2\alpha}$. In the contact process we can argue this by noting that the state at time
$t-L^{2\alpha}$ is dominated by the process starting from all 1 at time $t-2L^{2\alpha}$. Since $\xi \equiv 1$ is absorbing in evolutionary
games, this simple argument is not possible. 

\clearp

\section{Proof of Theorem \ref{odeconv}} \label{sec:pf4}

We use the notation for the dual introduced in Section \ref{sec:dualvmp}.
Let $R_n$ denote the increasing subsequence of jump times $T_m$ that are branching times, and $N(t)$ be the number that occur by time $t/\ep^2$.
Since branching times occur at rate $C_Bk\ep^2$ when there are $k$ particles in the dual, and the number of particles in the dual
is bounded by a branching process in which each particle gives birth to $M$ new particles at rate $C_B\ep^2$, the expected number of
particles at time $t/\ep^2$ is $\le \exp(C_BMt)$ times the number at time 0. 

Choose a time $K_L$ so that $K_L/L^2 \to \infty$ and $\ep_L^2K_L \to 0$. If there are $k$ particles at time 0, then
the second condition implies 
$$
P(R_1 \le K_L ) \to 0
$$
as $L\to\infty$. From this it follows easily that 
\beq
P( R_m - R_{m-1} \le K_L \hbox{ for some $m \le N(t)$} ) \to 0
\label{space1}
\eeq
as $L\to\infty$ and
\beq
P( t/\ep^2 - R_{N(t)} \le K_L ) \to 0.
\label{space2}
\eeq
Let $S_n = R_n + L^2 + K_L$.

The next three results concern the time intervals $[R_n,R_n+L^2]$, $[R_n+L^2,S_n]$ and $[S_n,R_{n+1}]$.

\mn
In the next three lemmas we suppose that at time 0 there are $k$ particles in the dual and no two particles are
within distance $L^{3/4}$ of each other.

\begin{lemma} \label{L2}
Suppose the first particle encounters a birth event in the dual at time 0.
With high probability (i) there is no coalescence between the newborn particles or with their parent after time $L^{2}$
and before the next birth event,
(ii) up to time $L^2$ there is no coalescence between the new born particles 
and particles 2, $\ldots$ k, and (iii) at time $L^2$ all particles are separated by $L^{3/4}$.
\end{lemma}

\begin{lemma} \label{L3} 
Even if we condition on the starting locations, then at time $K_L$
the particle locations at time $K_L$ are almost independent and uniformly distributed over the torus,
i.e., the total variation distance between the random positions and independent uniformly distributed positions tends to 0.
\end{lemma}

\begin{lemma} \label{L1}
With high probability, (i) there is no coalescence before the first time a birth event affects a particle,
and (ii) just before the first birth time the existing particles are separated by $L^{3/4}$.
\end{lemma}

\begin{proof}[Proof of Theorem \ref{odeconv}]
We will now argue that the three lemmas imply the desired conclusion. 
\eqref{space2} implies that with high probability there is no branching event in $[t/\ep^2-K_L,t/\ep^2].$
Let $N(t)$ be the last branching time before time $t/\ep^2$. By the last remark $S_{N(t)} < t/\ep^2$.
Lemma \ref{L1} implies that no coalescence occurs in the dual in $[S_{N(t)},t/\ep^2]$. 

When \eqref{space1} holds, Lemma \ref{L1} implies that the particles at time $R_{N(t)}$ are all separated by $L^{3/4}$ so using 
Lemma \ref{L2} all of the coalescences between the new born particles and their parent occur
before $R_{N(t)}+L^2$ and there is no coalescence with other particles during that time interval.
At time $R_N+L^2$ all the particles are separated by $L^{3/4}$, so Lemmas \ref{L3} and \ref{L1}
imply that there is no coalescence during $[R_{N(t)}+L^2,S_{N(t)}]$ and at time $S_{N(t)}$ the
particles are almost uniform over the torus.  

The results in the last paragraph imply that when the jump occurs at time $R_{N(t)}$ the joint distribution of the focal site and
its neighbors are approximately that of the voter model with the density at time $S_{N(t)}$. Since $\ep_L^2K_L \to 0$
this is almost the same as the density at time $R_{N(t)}$. Working backwards in time and using induction we see that
as $L\to\infty$ the dual on the time scale $t\ep_L^{-2}$ the dual converges to a limit in which branchings occur 
at rate $\|h\|_\infty$ and when they occur the joint distribution of the state of the focal site and its neighbors is given by the
density at time $t$.

The last observation implies $EU_i(t)$ converges to a limit $u_i(t)$. To show that the limit satisfies the differential equation
we consider $u_i(t+h)-u_i(t)$. When $h$ is small the the probability of two or more branching events is $o(h^2)$.
By considering the effect of a single event and letting $h\to 0$ we conclude that $du_i/dt = \phi_i(u)$.
For more details in the more complicated situation of convergence to a PDE, see Section 2 in Durrett and Neuhauser \cite{DN}
or Chapter 2 of \cite{CDP}.

The final detail is to show that $\var U_i(t)$. To do this, we note that if $x$ and $y$ are separated by
a distance greater than $L^{3/4}$ then with high probability their dual processes never intersect before time $t/\ep^2$.
Writing $\eta_t(x)$ as shorthand for $\xi^\ep_{t\ep^{-2}_L}(x)$, we see that if $\delta>0$ and $L$ is large
\begin{align*}
\var(U_i(t)) & = \frac{1}{N^2} \sum_{x \in \TL} \cov( \eta_t(x), \eta_t(y)) \\
& \le \frac{c_d L^{3d/4}}{L^d} + \delta
\end{align*}
Letting $L\to\infty$ and using Chebyshev's inequality gives the desired result.
\end{proof}

\subsection{Proofs of the three lemmas}

Let $p(x)$ be a finite range, irreducible symmetric random
walk kernel on $\ZZ^d$, some $d\ge 3$, with characteristic
function
\[
\phi(\theta) = \sum_{x\in\ZZ^d}e^{i\theta x}p(x), \quad
\theta\in [-\pi,\pi]^d
\]
Then $\phi$ is real-valued, and by PII.5 in Spitzer there is
a constant $\lambda>0$ such that
\begin{equation}\label{spitzer}
1-\phi(\theta)\ge \lambda |\theta|^2, \quad
\theta\in [-\pi,\pi]^d
\end{equation}

Let $Z_t$ 
be a rate-one continuous time walk on $\ZZ^d$ with
jump kernel $p$, starting at 0. 
Then $Z_t$ has characteristic function
\[
\phi_t(\theta) = E_0(e^{i\theta Z_t})=
\exp(-t(1-\phi(\theta)) ,
\]
and by (1), 
\begin{equation}\label{phitbnd}
\phi_t(\theta) \le 
e^{-\lambda t|\theta|^2}, \quad \theta\in [-\pi,\pi]^d
\end{equation}
Let $Z^L_t=Z_t\mod L$ be the corresponding walk
on the torus $\TL$. 

\begin{prop} \label{rw1}
(a) There is a constant $C>0$ such that
\begin{equation}\label{tLbnd}
P(Z^L_t=x) \le C\Big(L^{-d}\vee t^{-d/2}\Big), \quad t\ge 0, x\in\TL.
\end{equation}

(b) If $s_L\to\infty$ then
\begin{equation}\label{tvbnd}
\sup_{t\ge L^2s_L}\max_{x\in\TL}L^d|P(Z^L_t=x)-L^{-d}| = 0
\end{equation}
\end{prop}

\noindent
(b) implies Lemma \ref{L3}.

\begin{proof}
By a standard inversion formula,
\[
P(Z^L_t=x) = L^{-d}\sum_{y\in\TL}\phi_t(2\pi y/L)
e^{2\pi i xy/L}.
\]
Pulling out the $y=0$ term and using \eqref{phitbnd} gives
\begin{equation}\label{bnd0}
|P(Z^L_t=x) -L^{-d}|\le L^{-d}\sum_{y\in\TL, y\ne 0 }
e^{-\lambda t|2\pi y|^2/L^2}
\end{equation}
After bounding the sum by an integral and then changing
variables, there is a constant $C'>0$ such that 
\begin{equation}\label{intbnd}
\sum_{y\in\TL, y\ne 0 }
e^{-\lambda t|2\pi y|^2/L^2}  \le
C'\int_1^{L} 
e^{-\lambda t(2\pi r/L)^2} r^{d-1} dr
\le C' \Big(\frac{L}{2\pi\sqrt{\lambda t}}\Big)^d
\int_0^\infty u^{d-1}e^{-u^2} du.
\end{equation}
Since $C_d=\int_0^\infty r^{d-1}e^{-r}dr<\infty$,
\eqref{bnd0} and \eqref{intbnd} imply
\[
P(Z^L_t=x) \le L^{-d} + 
C_d C'
(2\pi{\lambda})^{-d}\ t^{-d/2}
\]
which proves part (a). 

For part (b),  plug any $t\ge L^2s_L$ into \eqref{bnd0}
and \eqref{intbnd} to get
\[
L^{d}|P(Z^x_{t}=x) -L^{-d}| \le
C_dC')(2\pi\sqrt{\lambda})^{-d}
s_L^{-d/2}.
\]
Since $s_L\to\infty$ this proves part (b).
\end{proof}

\begin{prop} (a) If $1\ll r_L\le L/2$ and $t_L\ll L^d$, then
\[
\max_{x\in\TL, |x|\ge r_L}P(Z^{L,x}_s=0\text{ for some }s\le
t_L) \to 0
\]

(b) If $r_L\ll L$  then
\[
\sup_{t \ge L^2} P(|Z^{L,0}_{t}|\le r_L)\to 0
\]
\end{prop}

\noindent
(b) with $r_L = L^{3/4}$ gives (iii) of Lemma \ref{L2} and (ii) of Lemma \ref{L1}.
(a) gives (i) of Lemma \ref{L1} and (ii) of Lemma \ref{L2}.
To prove (i) of Lemma \ref{L2}, combine (iii) of Lemma \ref{L2} with (i) of Lemma \ref{L1}.

\begin{proof}
Let $\tau=\inf\{t\ge0: Z^{L,x}_t=0\}$, let $|x|\ge r_L$, and
choose $\{t'_L\}$ such that $t'_L\to\infty$ and 
$t'_L\ll t_L\wedge r_L$. Using a standard martingale argument,
\[
P(\tau\in[0,t'_L]) \le P(\sup_{0\le t\le t'_L}|Z^{L,0}_t|\ge
r_L)
\le 2d P(|Z^0_{t'_L}|\ge r_L) \le
2d E(|Z^0_{t'_L}|)/r_L,
\]
which tends to 0 because $t'_L\ll r_L$. (Alternatively,
could use the assumption $p$ has finite range and
\eqref{tLbnd} for some $t'_L\to\infty$, which is all that's
needed for below.) 

Next, by the Markov property and the fact that $Z^{L,x}_t$ is a rate
one walk, \eqref{tLbnd} implies
\[
P(\tau\in[t'_L,t_L]) \le e\int_{t'_L}^{t_L+1}P(Z^x_t=0)
\le Ce\int_{t'_L}^{L^2\wedge t'_L}t^{-d/2}dt + 
Ce\int_{t'_L\wedge L^2}^{t'_L+1}L^{-d}dt.
\]
This tends to 0 because $t'_L\to\infty$ and $t'_L/L^d\to 0$,
proving (a).

For (b), the bound \eqref{tLbnd} implies
$P(Z^{L,0}_{L^2}=0)\le CL^{-d}$, so we have
\[
P(|Z^{L,0}_{L^2}|\le r_L) \le C L^{-d}(|r_L|+1))^d \to 0
\]
which is the desired result.
\end{proof}

\clearp

\section{Proof of Theorem \ref{TFth}} \label{sec:TFpf}

For this result we need to augment the construction with Poisson processes $T_n^x$, $n \ge 1$
that have rate $\mu$, and random variables $V_n^x$ that are uniform over the strategy set.
At time $T_n^x$ the value at $x$ is set equal to $V_n^x$. Since mutations tell us the value at a site,
when we work backward in the dual, we kill a particle when it encounters a mutation.
When all of the particles have been killed then we can compute the value of the process at all the 
sites used to initialize the dual.

Suppose first that there are no mutations. Since $w$ satisfies the conditions for Regime 2, it follows from
Theorem \ref{odeconv} that if we run time at rate $1/w$ then in the limit as $L \to\infty$ the density of type $k$ satisfies
\beq
\frac{du_k}{dt} = \phi_k(u).
\label{limnomut}
\eeq
If we now return to the case with mutations and assume that $\mu/w \to c$ then the limiting equations become
\beq
\frac{du_k}{dt} = \phi_k(u) + \frac{\mu}{w}(1/n - u_k)
\label{limDE}
\eeq
so equilibria are solutions of
$$
u_k = \frac{1}{n} + \frac{w}{\mu} \phi_k(u).
$$
Doing some algebra gives
$$
u_k - \frac{1}{n} = \frac{w}{\mu} \phi_k(1/n,\ldots 1/n) + \frac{w}{\mu}(\phi_k(u) - \phi_k(1/n,\ldots 1/n))
$$
and hence 
$$
\left|u_k - \frac{1}{n} - \frac{w}{\mu} \phi_k(1/n,\ldots 1/n) \right| \le \frac{w}{\mu}\left|\phi_k(u) - \phi_k(1/n,\ldots 1/n)\right|.
$$

Using the fact that $\phi_k$ is Lipschitz continuous we conclude 
\begin{align*}
|u_k - \frac{1}{n}| \le C_1w/\mu, &\\
\left|u_k - \frac{1}{n} - \frac{w}{\mu} \phi_k(1/n,\ldots 1/n) \right|& \le C_2 (w/\mu)^2.
\end{align*}
If $\mu/w \to\infty$ slowly enough then we can use the last result to conclude
$$
u_k > 1/n ,
$$
when $w$ is small giving the desired formula.

\clearp

\end{document}